\documentclass[12pt]{amsart}
\usepackage{epsfig,color}

\makeatother

\theoremstyle{definition}

\def\fnum{equation} 
\newtheorem{Thm}[\fnum]{Theorem}
\newtheorem{Cor}[\fnum]{Corollary}

\newtheorem{Lem}[\fnum]{Lemma}

\newtheorem{Def}[\fnum]{Definition}
\newtheorem{Exa}[\fnum]{Example}

\numberwithin{equation}{section}

\newcommand{\Vol}{{\text{Vol}}}
\newcommand{\V}{{\text{V}}}

\newcommand{\Ric}{{\text{Ric}}}

\newcommand{\Hess}{{\text {Hess}}}

\def\RR{{\bold R}}

\newcommand{\dv}{{\text {div}}}
\newcommand{\e}{{\text {e}}}

\newcommand{\cC}{{\mathcal{C}}}

\newcommand{\cF}{{\mathcal{F}}}
\newcommand{\cW}{{\mathcal{W}}}
\newcommand{\cL}{{\mathcal{L}}}

\newcommand{\eqr}[1]{(\ref{#1})}

\title[Monotonicity formulas I]{New monotonicity formulas for Ricci curvature and applications; I}

\author{Tobias Holck Colding}%
\address{MIT, Dept. of Math.\\
77 Massachusetts Avenue, Cambridge, MA 02139-4307.}

\thanks{The author
was partially supported by NSF Grant DM  
11040934 and NSF FRG grant DMS 
 0854774}

\email{colding@math.mit.edu}

\begin{document}

\maketitle

\begin{abstract}
We prove three new monotonicity formulas for manifolds with a lower Ricci curvature bound and show that they are connected to rate of convergence to tangent cones.  In fact, we show that the derivative of each of these three monotone quantities is bounded from below in terms of the Gromov-Hausdorff distance to the nearest cone.  The monotonicity formulas are related to the classical Bishop-Gromov volume comparison theorem and Perelman's celebrated monotonicity formula for the Ricci flow.  We will explain the connection between all of these.

Moreover, we show that these new monotonicity formulas are linked to a new sharp gradient estimate for the Green's function that we prove.  This is parallel to that Perelman's monotonicity is closely related to the sharp gradient estimate for the heat kernel of Li-Yau. 

In \cite{CM4} we will use the monotonicity formulas we prove here to show uniqueness of certain tangent cones of Einstein manifolds and in \cite{CM3} we will prove a number of related monotonicity formulas.

Finally, there are obvious parallels between our monotonicity and the positive mass theorem of Schoen-Yau and Witten.

\end{abstract}

\section{Introduction}

The results we will give holds for manifolds with any given lower bound for the Ricci curvature and are new and of interest both for small and large balls.  They are effective in the sense that the estimates we give do not depend on the particular manifold but only on some quantitative behavior like dimension and lower bound for Ricci curvature.  This allows us to pass these properties through to possible singular limits.  For simplicity we will concentrate our discussion on manifolds with nonnegative Ricci curvature and large balls though our results holds with obvious changes for small balls and any other fixed lower bound for the Ricci curvature.  Moreover, our results are local and holds even for balls in manifolds as long as the Ricci curvature is bounded from below on those balls.

A key property of Ricci curvature is monotonicity of ratio of volumes of balls.  For $n$-dimensional manifolds with nonnegative Ricci curvature Bishop-Gromov's volume comparison theorem, \cite{GLP}, \cite{G}, asserts that the relative volume
\begin{equation}  \label{e:bishopgromov1}
\Vol(r)=r^{-n}\,\Vol (B_r(x))\downarrow
\end{equation}
is monotone nonincreasing in the radius $r$ for any fixed $x\in M$.    As $r$ tend to $0$ this quantity on a smooth manifold converges to the volume of the unit ball in $\RR^n$ denoted by $\Vol (B_1(0))$ and as $r$ tends to infinity it converges to a nonnegative number $\V_M$.  If $\V_M>0$, then we say that $M$ has Euclidean volume growth.      An application of monotonicity of relative volume is Gromov's compactness theorem, \cite{GLP}, \cite{G}.   When $M$ has nonnegative Ricci curvature this compactness implies that any sequence of rescaling $(M,r_i^{-2}g)$, where $r_i\to \infty$ has a subsequence that converges in the Gromov-Hausdorff topology to a length space.   Any such limit is said to be a tangent cone at infinity of $M$.  

  A geometric property of Ricci curvature that will play a key role in the discussion below, both as a motivation and in some of the applications, comes from \cite{ChC1}.  It say that if $M$ has nonnegative Ricci curvature and $\Vol (r)$ is almost constant between say $r_0$ and $2r_0$, then the annulus is Gromov-Hausdorff close to a corresponding annulus in a cone.  
In particular, if $M$ has Euclidean volume growth, then any tangent cone at infinity of $M$ is a metric cone\footnote{Without the assumption of Euclidean volume growth tangent cones need not be metric cones by \cite{ChC2} and need not even be polar spaces by \cite{M2}.}.  In general our open manifolds of nonnegative Ricci curvature will be assumed to have faster than quadratic volume growth or more precisely be nonparabolic. 

A complete manifold is said to be nonparabolic if it admits a positive Green's function. Otherwise, it is said to be parabolic.  By a result of Varopoulos, \cite{V},   an open manifold with nonnegative Ricci curvature is nonparabolic if and only if 
\begin{align}
\int_1^{\infty} \frac{r}{\Vol (B_r(x))}\,dr<\infty\, .
\end{align}
When $M$ is nonparabolic, then we let $G$ be the minimal positive Green's function.  Combining the result of Varopoulos mentioned above with work of  Li-Yau, \cite{LY}, gives that if $M$ has nonnegative Ricci curvature and is nonparabolic, then for $x\in M$ fixed $G=G(x,\cdot)\to 0$ at infinity.  In other words, the function
\begin{align}
b=G^{\frac{1}{2-n}}
\end{align}
is well defined and proper; cf. \cite{CM1}, \cite{CM2}. 

To put our results in perspective we will briefly recall some of the most relevant monotonicity formulas for the current discussion.   

The Bishop-Gromov volume comparison theorem, \cite{GLP}, \cite{G}, was described above.   It asserts that the ratio of volume of a ball in a manifold with nonnegative Ricci curvature centered at a fixed point to the volume of a Euclidean ball the same radius is monotone nonincreasing in the radius.  This parallels the monotonicity for minimal surfaces where the same quantity is monotone; however for minimal surfaces the ratio is monotone nondecreasing.  Moreover, for minimal surfaces balls are intersections of extrinsic balls with the surface as opposed to intrinsic balls in the Bishop-Gromov.    Either of these monotonicity formulas follows from integrating the Laplacian of the distance squared to a point.  In one case it is the extrinsic distance; in the other the intrinsic distance.  In fact, in all of the monotonicity formulas that we discuss below monotonicity will come from integrating the Laplacian of appropriately chosen functions.

For mean curvature flow an important monotone quantity was found by Huisken, \cite{H}.  Huisken integrated a backward extrinsic heat kernel over the evolving hypersurface and showed that under the mean curvature flow this quantity is monotone nonincreasing.  This is a parabolic monotonicity where the backward heat kernel is integrated over the entire evolving hypersurfaces and thus the quantity is global.

For Ricci flow Perelman found two new quantities, the $\cF$ and $\cW$ functional, and proved that $\cW$ is monotone, \cite{P1}.  Even for static solutions to the Ricci flow, that is, for Ricci flat manifolds, the $\cF$ and $\cW$ functionals are interesting and the monotonicity of $\cW$ is nontrivial. In fact, if one omit the scalar curvature term in the $\cW$-functional, then it is even monotone for manifolds with nonnegative Ricci curvature as was pointed out by Lei Ni, \cite{N1}-\cite{N3}. Moreover, as was known already to Perelman, the monotonicity of $\cW$ is related to both a sharp log Sobolev inequality and a sharp gradient estimate for the heat kernel $H$.  Because of this it is instructive to first recall the sharp gradient estimate of Li-Yau, \cite{LY}.  This asserts that on a manifold with nonnegative Ricci curvature 
\begin{align}
t\,\left(\frac{|\nabla H|^2}{H^2}-\frac{H_t}{H}\right)-\frac{n}{2}=-t\,\Delta\,\log H -\frac{n}{2}\leq 0\ .
\end{align}
Integrating this over the manifold against the heat kernel as a weight gives the $\cF$-functional for Ricci flat manifolds and what we call the $F$-function on a fixed manifold with nonnegative Ricci curvature, see Lei Ni, \cite{N1}-\cite{N3}, 
\begin{align}
F(t)&=t\int_M\left(\frac{|\nabla H|^2}{H^2}-\frac{H_t}{H}\right)\,H\,d\Vol-\frac{n}{2}=-t\int_M\Delta\,\log H\,H\,d\Vol -\frac{n}{2}\notag\\
&=t\int_M|\nabla \log H|^2\,Hd\Vol -\frac{n}{2}\leq 0\, .
\end{align}
Here the last equality comes from integration by parts.  Note that the $F$-function is a function of two variables: $t$ and the `center' $x$ though usually $x$ is fixed in which case we think of it as a function only of $t$.  The dependence of $x$ comes from the heat kernel $H=H(x,\cdot,t)$.  It is not hard to see that $\frac{F}{t}$ is the derivative of the Shannon type\footnote{$S$ is also sometimes referred to as the Nash entropy; see, for instance \cite{N1}-\cite{N3}.} entropy\footnote{We use a slightly different normalization in both $F$ and $S$ than the standard one; however this normalization does not affect $W$.  Our normalization is chosen so that on Euclidean space both $F$ and $S$ vanishes.}
\begin{align}
S(t)=-\int_M\log H\,H\,d\Vol-\frac{n}{2}\,\log (4\pi\,t)-\frac{n}{2}\, .
\end{align}
Perelman went on and defined 
\begin{align}
W=F+S
\end{align} 
and showed that 
\begin{align}
W\downarrow
\end{align}
is monotone nonincreasing; cf. Lei Ni, \cite{N1}-\cite{N3} and Section \ref{s:perelman} where we discuss these quantities in greater detail.

Our three new monotonicity theorems, see Section \ref{s:monotone} for the precise statements, comes from a new sharp gradient estimate for the Green's function $G$ on manifolds with nonnegative Ricci curvature.  This new sharp gradient estimate asserts that $b=G^{\frac{1}{2-n}}$ satisfies
\begin{align}  \label{e:newgradient}
|\nabla b|^2-1=\frac{\Delta b^2}{2n}-1\leq 0\, ;
\end{align}
see Theorem \ref{t:sharpGreen}.   Moreover, if at one point in $M\setminus \{x\}$ we have equality in this inequality, then the manifold is flat Euclidean space.    In addition, we also show a sharp asymptotic gradient estimate of $b$ for $r\to \infty$; see Theorem \ref{t:asympsharpGreen}.
Integrating \eqr{e:newgradient} over the level sets of $b$ against the weight $r^{1-n}\,|\nabla b|$ gives our basic new quantity $A$ that in our (elliptic) monotonicity formulas plays the role of Perelman's $F$-function.  Namely, set 
\begin{align}
A(r)&=r^{1-n}\int_{b=r}\left(|\nabla b|^2-1\right)\,|\nabla b|=r^{1-n}\int_{b=r} |\nabla b|^3-\Vol (\partial B_1(0))\\
&=r^{1-n}\int_{b=r} |\nabla b|^3-\Vol (\partial B_1(0))=\frac{r^{1-n}}{2n}\int_{b=r} \Delta b^2\,|\nabla b|-\Vol (\partial B_1(0))\leq 0\, ;\notag
\end{align}
where $B_1(0)\subset \RR^n$ is the unit ball.  NOTE that in the main body of this paper $A$ and $V$ DIFFERS from the ones defined here in the introduction by the constants $\Vol (\partial B_1(0))$ and $\Vol (B_1(0))$ respectively, as in the later sections we have NOT subtracted their Euclidean values.  All of our monotonicity formulas involves $A$.  In particular, we show that 
\begin{align}
A\downarrow\text{ and }V\downarrow
\end{align}
are monotone nonincreasing (see Corollary \ref{c:properties}; and also Theorem \ref{t:supnoninc} for a related statement), where 
\begin{align}
V(r)=r^{-n}\int_{b\leq r}|\nabla b|^4-\Vol (B_1(0))\, .
\end{align}
The monotonicity of $V$ is parallel to the Bishop-Gromov volume comparison theorem as stated in \eqr{e:bishopgromov1}.  The standard proof of the Bishop-Gromov volume comparison (using the Laplacian comparison theorem applied to the distance to a fixed point) go over first showing that the ratio
\begin{align}   \label{e:bishopgromov2}
r^{1-n}\Vol (\partial B_r(x))\downarrow
\end{align}
is monotone nonincreasing and then use this to show \eqr{e:bishopgromov1}.  The monotonicity of $A$ is the parallel of \eqr{e:bishopgromov2}.
Note that it follows easily from the coarea formula, see Lemma \ref{l:coarea}, that $r\,V'=A-n\,V$; thus the monotonicity of $V$ implies that $A\leq n\,V$.  Therefore an interesting (and natural) question would for instance be wether or not the gap between $A$ and $n\,V$ widens.  The monotonicity of both $A$ and $V$ are byproducts of our main monotonicity theorems.
The first of our three main monotonicity formulas, see Theorem \ref{t:mainmono}, show that 
\begin{align}
2(n-1)V-A\downarrow
\end{align} 
is monotone nondecreasing and gives an exact (and useful) formula for the derivative.
  
The monotonicity of Perelman's $W$-function is easily seen to be equivalent to that $t\,F$ is monotone; this follows from that $W=F+S$ and $S'=\frac{F}{t}$.  The parallel to this in our setting is that one of our monotonicity formulas, see Theorem \ref{t:secondmono1},  asserts that 
\begin{align}
r^{2-n}\,A\uparrow
\end{align}
 is monotone nondecreasing.

One of the major points of this article is that not only are the quantities we define monotone, but their derivatives are something useful that monotonicity helps us bound.  In fact, this was the starting point of this article and is one of the advantages of our new monotonicity formulas compared with say the classical Bishop-Gromov volume comparison theorem.  For instance, for the derivative of $2(n-1)\, V-A$ we show in Theorem \ref{t:mainmono} that
\begin{align}
       \left[2(n-1)\, V-A\right]'(r)= -\frac{1}{2}\,r^{-1-n}\int_{b\leq r}\left(\left|\Hess_{b^2}-\frac{\Delta b^2}{n}\,g\right|^2+\Ric (\nabla b^2,\nabla b^2)\right)\, ,
       \end{align}
       and for the derivative of $r^{2-n}\,A$ we show in Theorem \ref{t:secondmono1} that
       \begin{align}    
       \left(r^{2-n} \, A\right)'(r)=\frac{r^{1-n}}{2}\int_{b\leq r}\left(\left|\Hess_{b^2}-\frac{\Delta b^2}{n}\,g\right|^2+\Ric (\nabla b^2,\nabla b^2)\right)
       \,b^{-n}\, .
       \end{align}
       In addition to these new monotonicity formulas and gradient estimates for the Green's function, then we estimate from below the derivative of our formulas in terms of the Gromov-Hausdorff distance to the nearest cone; see Theorems \ref{t:fund2a} and \ref{t:fund3a}.  
For instance, loosely speaking, Theorem \ref{t:fund2a} shows that
\begin{align}
      -C\, \left(2(n-1)\, V-A\right)'(r)\geq  \frac{\Theta_r^{2}}{r}\, ;
       \end{align}
where $\Theta_r$ is the scale invariant Gromov-Hausdorff distance between $B_r(x)$ and the ball of radius $r$ in the nearest cone centered at the vertex.   The constant $C$ depends on the dimension of the manifold, the lower bound for the Ricci curvature, and also on a positive lower bound for the volume of $B_r(x)$.    The actual statement of the theorem is slightly more complicated as in reality the right hand side of this inequality is not to the power $2$ rather to the slightly worse power $2+2\epsilon$ for any $\epsilon> 0$, and the constant $C$ also depends on $\epsilon$; cf. \cite{CN1}.
We prove this lower bound, see Theorem \ref{t:fund} and Corollary \ref{c:fund}, for the derivative using \cite{ChC1}; cf. also \cite{C1}--\cite{C3}.         
We also prove a similar lower bound for the derivative of Perelman's $F$-function though in that case it is a weighted distance to the nearest cone where the nearest cone is allowed to change from scale to scale.

In \cite{CM4} we will use the monotonicity formulas we prove here to show uniqueness of certain tangent cones of Einstein manifolds.  In a second paper \cite{CM3} we will show further monotonicity formulas and discuss other applications.

Finally, we note that one may think of our new monotonicity formulas as enhanced versions of the classical Bishop-Gromov volume comparison theorem.

  \section{Monotonicity formulas}  \label{s:monotone}
  
  In this section $M^n$ will be a smooth complete $n$-dimensional manifold where $n\geq 3$.  We will later be particularly interested in the case where $M$ has nonnegative Ricci curvature, however the computations that follows holds on any smooth manifold.

  Suppose that $G$ is the Green's function\footnote{Our Green's functions will be normalized so that on Euclidean space of dimension $n\geq 3$ the Green's function is $r^{2-n}$} on a manifold $M$; fix $x\in M$ and set $G=G_x=G(x,\cdot)$.  One sometimes say that $G=G_x$ is the Green's function with pole at $x$.  Following \cite{CM1}, \cite{CM2}, we set\footnote{The normalization of $b$ that we use here differs from that used in \cite{CM1}, \cite{CM2} by a constant.}  
  \begin{align}
  b=G^{\frac{1}{2-n}}\, ;
     \end{align}
     then 
 \begin{align}
 \Delta b^2=2n\,|\nabla b|^2\, .
     \end{align}
     
     We will use a number of times below that if, as in  \cite{CM1}, \cite{CM2},  we set
  \begin{align}
  I_v(r)=r^{1-n}\int_{b=r}v\,|\nabla b|=\frac{1}{n-2}\int_{b=r}v\,|\nabla G|\, ,
     \end{align}
     then 
  \begin{align}
  I'_v=r^{1-n}\int_{b=r}v_n=r^{1-n}\int_{b\leq r}\Delta\, v\, .
     \end{align}
     Here $v_n$ is the (outward) normal derivative of the function $v$; normal to the boundary of $\{b\leq r\}$.  In particular, the function 
      \begin{align}
  I_1(r)=r^{1-n}\int_{b=r}|\nabla b|\, ,
     \end{align}
     is constant in $r$.

\subsection{The `area' and the `volume'}     
     
     Define nonnegative functions $A(r)$ and $V(r)$ by
   \begin{align}
  A(r)&=r^{1-n}\int_{b=r}|\nabla b|^3\, ,\\
     V(r)&=r^{-n} \int_{b\leq r}|\nabla b|^4\, .
  \end{align}
  Note that these quantities differs from the ones we defined in the introduction by constants since here, unlike the introduction, we have not subtracted their Euclidean values
  
   The next simple lemma will be used three places:  First to compute the limit of $A(r)$ and $I_1$ as $r\to 0$; second in the proof of the second monotonicity theorem where the lemma also enters via the same limit of $A$ and third in the sharp gradient estimate for the Green's function.
    
\begin{Lem}  \label{l:ref}
Let $M$ be a smooth manifold with $n\geq 3$, then
\begin{align}
\lim_{r\to 0}\sup_{\partial B_r(x)}\left|\frac{b}{r}- 1\right|&=0\, , \label{e:ref1}\\
\lim_{r\to 0}\sup_{\partial B_r(x)}\left||\nabla b|^2-1\right|&= 0\, ,\label{e:ref2}\\
\lim_{r\to 0}A(r)=\lim_{r\to 0}I_1(r)&=\Vol (\partial B_1(0))\, ,\label{e:ref3}\\
\lim_{r\to 0}V(r)&=\Vol (B_1(0))\, .\label{e:ref4}
\end{align}
\end{Lem}

\begin{proof}
In \cite{GS} it was shown that for the Green's function with pole at $x$ 
\begin{align}
G(y)&=d^{2-n}(x,y)\,(1+o(1))\, ,\\
|\nabla G(y)|&=(n-2)\,d^{1-n}(x,y)\,(1+o(1))\, ,
\end{align}
where $o(1)$ is a function with $o(y)\to 0$ as $y\to x$. From this the first two clams easily follows.

To see \eqr{e:ref3} observe first that it follows from \eqr{e:ref2} that $I$ and $A$ has the same limit.  It is therefore enough to show that the limit of $I$ is $=\Vol (\partial B_1(0))$.  To see this use that $I$ is constant in $r$ together with the coarea formula to rewrite $I$ as 
 \begin{align}
     \int_{b\leq r}|\nabla b|^2=\int_{0}^r\int_{b= s}r^{n-1}\,I_1(s)=\frac{r^{n}\,I_1(1)}{n}\, ,
  \end{align}
  and thus
  \begin{align}
    \lim_{r\to 0}I_1(r)=I_1(1)=n\, \lim_{r\to 0} r^{-n}\int_{b\leq r}|\nabla b|^2=\Vol (\partial B_1(0))\, .
  \end{align}
 Here the last equality followed from \eqr{e:ref1} and \eqr{e:ref2}.
 
 Finally, \eqr{e:ref4} follows easily from \eqr{e:ref1} and \eqr{e:ref2}.
\end{proof}
  
  Moreover, we have the following:
  
  \begin{Lem}  \label{l:coarea}
 \begin{align}
    V'(r)=\frac{1}{r}\left(A(r)-n\,V(r)\right)\, .
       \end{align}
  \end{Lem}
  
  \begin{proof}
  By the coarea formula we can rewrite $V(r)$ as 
 \begin{align}
     V(r)=r^{-n} \int_{-\infty}^r\int_{b= s}|\nabla b|^3\, .
  \end{align}
  From this the lemma easily follows.
  \end{proof}
  
  We will later see that on any manifold with nonnegative Ricci curvature the gradient of $b$ is bounded by some universal constant depending only on the dimension; see Lemma \ref{l:cLsub1} for a gradient bound and Theorem \ref{t:sharpGreen} for the eventual sharp bound.  Together with the next lemma this implies that both $A$ and $V$ are bounded.  We will then come back later and use our main monotonicity theorem to show that both $A$ and $V$ are monotone nonincreasing on a manifold with nonnegative Ricci curvature and hence, in particular, they are bounded by their values as $r\to 0$.  
  
  \begin{Lem}
  If $|\nabla b|\leq C$, then 
   \begin{align}
       A&\leq C^2\, \Vol (\partial B_1(0))\, ,\\
         V&\leq \frac{C^2}{n}\, \Vol (\partial B_1(0))=C^2\, \Vol (B_1(0))\, .
       \end{align}
    \end{Lem}
    
      \begin{proof}
      The first claim follows from that $I_1$ is constant as a function of $r$ and that we found in Lemma \ref{l:ref} what that constant is.  The second claim follows from the first together with that
      \begin{align}  \label{e:coareaforV}
         V(r)=r^{-n} \int_0^rs^{n-1}\,A(s)\, .
       \end{align}
        \end{proof}
  
\subsection{The first monotonicity formula}

  Our first monotonicity result is the following:
       
       \begin{Thm}  \label{t:mainmono}
       \begin{align}
       \left(A-2(n-1)\, V\right)'= \frac{1}{2}r^{-1-n}\int_{b\leq r}\left(\left|\Hess_{b^2}-\frac{\Delta b^2}{n}\,g\right|^2+\Ric (\nabla b^2,\nabla b^2)\right)\, .
       \end{align}
       \end{Thm}    
     
     \begin{proof}
     Observe first that we can trivially rewrite $A(r)$ as follows
      \begin{align}
  A(r)&=r^{1-n}\int_{b=r}|\nabla b|^3=\frac{1}{4}r^{-1-n}\int_{b=r}|\nabla b^2|^2\, |\nabla b|\, . 
  \end{align}

     Computing gives
        \begin{align}
  r^{-2}\,\left(r^2\,A\right)'(r)&
  =\frac{1}{4}r^{-1-n}\int_{b= r}\frac{d}{dn} |\nabla b^2|^2=\frac{1}{4}r^{-1-n}\int_{b\leq r}\Delta |\nabla b^2|^2\notag\\
  &=\frac{1}{2}r^{-1-n}\int_{b\leq r}\left(|\Hess_{b^2}|^2+\langle \nabla \Delta b^2,\nabla b^2\rangle+\Ric (\nabla b^2,\nabla b^2)\right)\notag\\
  &=\frac{1}{2}r^{-1-n}\int_{b\leq r}\left(|\Hess_{b^2}|^2-|\Delta b^2|^2+\Ric (\nabla b^2,\nabla b^2)\right)+\frac{1}{2}r^{-1-n}\int_{b=r}\Delta b^2\,  \frac{d}{dn}b^2\\
 &=\frac{1}{2}r^{-1-n}\int_{b\leq r}\left(|\Hess_{b^2}|^2-|\Delta b^2|^2+\Ric (\nabla b^2,\nabla b^2)\right)+2n\, r^{-n}\int_{b=r}|\nabla b|^3\notag\\
 &=\frac{1}{2}r^{-1-n}\int_{b\leq r}\left(|\Hess_{b^2}|^2-|\Delta b^2|^2+\Ric (\nabla b^2,\nabla b^2)\right)+\frac{2n}{r} A(r)\notag \, .
     \end{align}
     Moreover, 
     \begin{align}
  \left|\Hess_{b^2}-\frac{\Delta b^2}{n}\,g\right|^2&= \left|\Hess_{b^2}\right|^2+\frac{|\Delta b^2|^2}{n}-2\frac{|\Delta b^2|^2}{n}\\
  &=\left|\Hess_{b^2}\right|^2-\frac{|\Delta b^2|^2}{n}\, .\notag
     \end{align}
     Hence,
     \begin{align}
     \left|\Hess_{b^2}\right|^2-|\Delta b^2|^2&= \left|\Hess_{b^2}-\frac{\Delta b^2}{n}\,g\right|^2-\left(1-\frac{1}{n}\right)\,|\Delta b^2|^2\\
     &=\left|\Hess_{b^2}-\frac{\Delta b^2}{n}\,g\right|^2-4n^2\,\left(1-\frac{1}{n}\right)\,|\nabla b|^4\, .\notag
   \end{align}
 Inserting this in the above gives
  \begin{align}
   r^{-2}\,\left(r^2\,A\right)'(r)&=\frac{1}{2}r^{-1-n}\int_{b\leq r}\left(\left|\Hess_{b^2}-\frac{\Delta b^2}{n}\,g\right|^2+\Ric (\nabla b^2,\nabla b^2)\right)\\
  &-2\left(1-\frac{1}{n}\right)\,n^2r^{-1-n}\int_{b\leq r}|\nabla b|^4+\frac{2n}{r} A(r)\notag \, .
     \end{align}
       Using lemma \ref{l:coarea} we can now rewrite the above as 
 \begin{align}
  r^{-2}\,\left(r^2\,A\right)'(r)&=\frac{1}{2}r^{-1-n}\int_{b\leq r}\left(\left|\Hess_{b^2}-\frac{\Delta b^2}{n}\,g\right|^2+\Ric (\nabla b^2,\nabla b^2)\right)
  -\frac{2\left(1-\frac{1}{n}\right)\,n^2}{r}V(r)+\frac{2n}{r} A(r)\notag\\
  &=\frac{1}{2}r^{-1-n}\int_{b\leq r}\left(\left|\Hess_{b^2}-\frac{\Delta b^2}{n}\,g\right|^2+\Ric (\nabla b^2,\nabla b^2)\right)+\frac{2n}{r} \left(A(r)-n\,V(r)\right)+\frac{2n}{r} V(r) \, .
       \end{align}
       Or, equivalently, since $r^{-2}\,\left(r^2\,A\right)'=A'+\frac{2}{r}A$
  \begin{align}
  A'=\frac{1}{2}r^{-1-n}\int_{b\leq r}\left(\left|\Hess_{b^2}-\frac{\Delta b^2}{n}\,g\right|^2+\Ric (\nabla b^2,\nabla b^2)\right)+\frac{2(n-1)}{r} \left(A-n\,V\right)\, .
       \end{align}     
       Therefore
       \begin{align}
       \left(A-2(n-1)\, V\right)'= \frac{1}{2}r^{-1-n}\int_{b\leq r}\left(\left|\Hess_{b^2}-\frac{\Delta b^2}{n}\,g\right|^2+\Ric (\nabla b^2,\nabla b^2)\right)\, .
       \end{align}
       \end{proof}

       In particular, on a manifold with nonnegative Ricci curvature we get the following:
       
       \begin{Cor}
       If $M$ is an $n$-dimensional manifold with nonnegative Ricci curvature, then for all $r>0$
       \begin{align}
       A(r)-\Vol (\partial B_1(0))\geq 2(n-1)\, \left(V(r)-\Vol (B_1(0))\right)\, .
       \end{align}
       
       Moreover, if for some $r>0$ we have equality, then the set $\{b\leq r\}$ is isometric to a ball of radius $r$ in $\RR^n$. 
       \end{Cor}
       
       \begin{proof}
       The inequality follows trivially from the theorem.   Suppose therefore that for some $r>0$ we have equality. 
       Since $M$ has nonnegative Ricci curvature, then by Theorem \ref{t:mainmono}
       \begin{align}
       \Hess_{b^2}&=\frac{\Delta b^2}{n}\,g\, ,\\
       \Ric (\nabla b^2,\nabla b^2)&=0\, .
        \end{align}
        From this it now follows from section 1 of \cite{ChC1} that $\{b\leq r\}$ is a metric cone and that $b$ is the distance to the vertex.  Since Euclidean space is the only smooth cone the corollary follows.
       \end{proof}
       
       Note that the inequality in the above corollary goes in the opposite direction of the usual Bishop-Gromov volume comparison theorem for manifolds with nonnegative Ricci curvature where the scale invariant volume of the boundary of a ball is bounded by the inside.  This is closely connected with that the above inequality deals with the excess relative to the Euclidean quantities.
       
       Likewise we get:
       
       \begin{Cor}
       If $M$ is an $n$-dimensional manifold with nonnegative Ricci curvature and $r_2>r_1>0$, then 
       \begin{align}
       A(r_2)-2(n-1)V(r_2)\geq A(r_1)-2(n-1)V(r_1)\, 
       \end{align}
       and equality holds if and only if the set $\{b\leq r_2\}$ is isometric to a ball of radius $r_2$ in Euclidean space. 
       \end{Cor}

       Note also that all of the above computations works for any positive harmonic function $G$ with $\frac{1}{G}$ proper and not necessarily the Green's function.
              
       \subsection{The second monotonicity formula}
       
              The next lemma holds for any positive harmonic function $G$, where, as before, $b$ is given by that $b^{2-n}=G$.
       
       \begin{Lem}   \label{e:LaplE}
       \begin{align}  \label{e:bochner3}
       b^2\, \Delta |\nabla b|^2+(2-n)\,\langle \nabla b^2,\nabla |\nabla b|^2\rangle
       &= \frac{1}{2}\left(\left|\Hess_{b^2}-\frac{\Delta b^2}{n}g\right|^2+\Ric (\nabla b^2,\nabla b^2)\right)\, ,\\
       \Delta \left(|\nabla b|^2\, G\right)&=\frac{1}{2}\left(\left|\Hess_{b^2}-\frac{\Delta b^2}{n}g\right|^2+\Ric (\nabla b^2,\nabla b^2)\right)\,b^{-n}\, .\label{e:bsub}
       \end{align}
       \end{Lem}

       \begin{proof}
       By the Bochner formula, as in the proof of Theorem \ref{t:mainmono},
        \begin{align}  \label{e:bochner1}
       \frac{1}{2}\Delta |\nabla b^2|^2&=|\Hess_{b^2}|^2+\langle \nabla \Delta b^2,\nabla b^2\rangle +\Ric (\nabla b^2,\nabla b^2)\notag\\
       &=\left|\Hess_{b^2}-\frac{\Delta b^2}{n}g\right|^2+\frac{|\Delta b^2|^2}{n}+\langle \nabla \Delta b^2,\nabla b^2\rangle +\Ric (\nabla b^2,\nabla b^2)\\
       &=\left|\Hess_{b^2}-\frac{\Delta b^2}{n}g\right|^2+4n\,|\nabla b|^4+2n\,\langle \nabla |\nabla b|^2,\nabla b^2\rangle +\Ric (\nabla b^2,\nabla b^2)\, .\notag
       \end{align}
       Moreover, since
        \begin{align}
       |\nabla b^2|^2=4\, b^2\, |\nabla b|^2\, .
       \end{align}
       we have that
       \begin{align}  \label{e:bochner2}
       \Delta |\nabla b^2|^2&=4\, b^2\, \Delta |\nabla b|^2+4\, \Delta b^2\, |\nabla b|^2+8\,\langle \nabla b^2,\nabla |\nabla b|^2\rangle\\
       &=4\, b^2\, \Delta |\nabla b|^2+8n\, |\nabla b|^4+8\,\langle \nabla b^2,\nabla |\nabla b|^2\rangle\notag\, .
       \end{align}
       Combining \eqr{e:bochner1} with \eqr{e:bochner2} gives \eqr{e:bochner3}.
       
       To see the second claim use Leibniz' rule and \eqr{e:bochner3} to get
        \begin{align}  
       2 \, \Delta \left(|\nabla b|^2\, G\right)&=2 \, b^2\, G\,\Delta |\nabla b|^2+(4-2n)\,\langle \nabla b^2,\nabla |\nabla b|^2\rangle\,  b^{1-n}\notag\\
       &=b^{-n}\,\left(2b^2\,\Delta\, |\nabla b|^2+(4-2n)\,\langle \nabla b^2,\nabla |\nabla b|^2\rangle\right)\\
       &=\left(\left|\Hess_{b^2}-\frac{\Delta b^2}{n}g\right|^2+\Ric (\nabla b^2,\nabla b^2)\right)\,b^{-n}\, .\notag
       \end{align}
    \end{proof}

       Lemma \ref{e:LaplE} also lead us directly to our second monotonicity formula:
       
          \begin{Thm}  \label{t:secondmono1}
       \begin{align}   \label{e:secondmono1e1}
       (2-n)\,(A-\Vol (\partial B_1(0)))+r\,A'=\frac{1}{2}\int_{b\leq r}\left(\left|\Hess_{b^2}-\frac{\Delta b^2}{n}g\right|^2+\Ric (\nabla b^2,\nabla b^2)\right)\,b^{-n}\, .
        \end{align}
        Or, equivalently, 
       \begin{align}    \label{e:secondmono1e2}
       \left(r^{2-n} \, \left[A-\Vol (\partial B_1(0))\right]\right)'=\frac{r^{1-n}}{2}\int_{b\leq r}\left(\left|\Hess_{b^2}-\frac{\Delta b^2}{n}\,g\right|^2+\Ric (\nabla b^2,\nabla b^2)\right)
       \,b^{-n}\, .
       \end{align} 
        \end{Thm}
       
          \begin{proof}
          For $r_2>r_1>0$ by Stokes' theorem and Lemma \ref{e:LaplE}
        \begin{align}
        r_2^{n-1}\, \left(r^{2-n}\,A\right)'(r_2&)-r_1^{n-1}\, \left(r^{2-n}\,A\right)'(r_1)=r_2^{n-1}\, I_{|\nabla b|^2\,G}'(r_2)-r_1^{n-1}\, I_{|\nabla b|^2\,G}'(r_1)\notag\\
       &=\int_{b=r_2}\left(|\nabla b|^2\, G\right)_n- \int_{b=r_1}\left(|\nabla b|^2\, G\right)_n\\
       &=\frac{1}{2}\int_{r_1\leq b\leq r_2}\left(\left|\Hess_{b^2}-\frac{\Delta b^2}{n}g\right|^2+\Ric (\nabla b^2,\nabla b^2)\right)\,b^{-n}\, .\notag
       \end{align}
       Since
       \begin{align}
       r^{n-1}\, \left(r^{2-n}\,A\right)'=(2-n)\,A+r\,A'\, ,
        \end{align}
        and, as we will see shortly, there exists a sequence $r_i\to 0$ 
         \begin{align}  \label{e:ref}
       (2-n)\,A(r_i)+r_i\,A'(r_i)\to (2-n)\,\Vol (\partial B_1(0))\text{ as }r_i\to 0\, ,
        \end{align}
        we get that
        \begin{align}  \label{e:second2}
       (2-n)\,(A-\Vol (\partial B_1(0)))+r\,A'=\frac{1}{2}\int_{b\leq r}\left(\left|\Hess_{b^2}-\frac{\Delta b^2}{n}g\right|^2+\Ric (\nabla b^2,\nabla b^2)\right)\,b^{-n}\, .
        \end{align}
        To see \eqr{e:ref} we need Lemma \ref{l:ref}.  Namely, by \eqr{e:ref3} $A(r)\to \Vol (\partial B_1(0))$ as $r\to 0$.  Moreover, it follows from this that $A$ is uniformly bounded for $r$ sufficiently small and hence there exists a sequence $r_i\to 0$ so that $r_i\,A'(r_i)\to 0$.  
          \end{proof}

       We can also reformulate this second monotonicity theorem by defining a second `volume of balls'.  We do that by setting
       \begin{align} \label{e:definfty1}
       V_{\infty}=\int_{1\leq b\leq r} \left(|\nabla b|^2-1\right)\,|\nabla b|^2\,b^{-n}\, .
       \end{align}
       So that by the coarea formula
       \begin{align} \label{e:definfty2}
       V_{\infty}=\int_1^rs^{-n}\int_{b=s}\left(|\nabla b|^3-|\nabla b|\right)\, ,
       \end{align}
       and hence
       \begin{align}  \label{e:diff}
              V_{\infty}'=r^{-n}\int_{b=r}\left(|\nabla b|^3-|\nabla b|\right)=\frac{A-\Vol (\partial B_1(0))}{r}\, .
       \end{align}
       Note that when $r<1$ the integral \eqr{e:definfty1} is interpreted as \eqr{e:definfty2}.
       It is not clear that this new $V_{\infty}$ is bounded even for manifolds with nonnegative Ricci curvature and indeed we will show that in general it is not.
       
       We can now reformulate our second monotonicity theorem in terms of this second `volume of balls' as follows:
       
       \begin{Thm}  \label{t:secondmono2}
       \begin{align}
       \left(A-(n-2)\,V_{\infty}\right)'=\frac{1}{2r}\int_{b\leq r}\left(\left|\Hess_{b^2}-\frac{\Delta b^2}{n}g\right|^2+\Ric (\nabla b^2,\nabla b^2)\right)\,b^{-n}\, .
        \end{align}
       \end{Thm}
       
          \begin{proof}
         This follows from \eqr{e:secondmono1e1}.
          \end{proof}

          Similar to the situation after the first monotonicity formula we get the following immediate corollary from this second monotonicity for manifolds with nonnegative Ricci curvature (the proof is with obvious changes the same as in the earlier corollaries of the first monotonicity formula).

       \begin{Cor}
       If $M$ is an $n$-dimensional manifold with nonnegative Ricci curvature and $r_2>r_1>0$, then 
       \begin{align}
       A(r_2)-(n-2)\, V_{\infty}(r_2)\geq A(r_1)-(n-2)\, V_{\infty}(r_1)\, 
       \end{align}
       and equality holds if and only is the set $\{b\leq r_2\}$ is isometric to a ball of radius $r_2$ in Euclidean space. 
       \end{Cor}

        \begin{Thm}
        Set $J(s)=-(n-2)\,s\,V_{\infty}(s^{\frac{1}{2-n}})$, then
        \begin{align}
       J'&=A-\Vol (\partial B_1(0))-(n-2)\,V_{\infty}\, \\
       J''(s)&=-\frac{1}{2(n-2)s}\int_{b\leq s^{\frac{1}{2-n}}}\left(\left|\Hess_{b^2}-\frac{\Delta b^2}{n}g\right|^2+\Ric (\nabla b^2,\nabla b^2)\right)\,b^{-n}\, .
       \end{align}
       \end{Thm}
       
        \begin{proof}
        This follows from a straightforward computation combined with Theorem \ref{t:secondmono2}.  
        \end{proof}
        
        We next use \cite{CM2} to calculate the asymptotic description of $A$ and $V$ for manifolds with nonnegative Ricci curvature:
        
        \begin{Thm}  \label{t:asymptotic}
       If $M^n$ has nonnegative Ricci curvature, then
       \begin{align}
       \lim_{r\to \infty} \frac{A(r)}{\Vol (\partial B_1(0))}&=\left(\frac{\V_M}{\Vol (B_1(0))}\right)^{\frac{2}{n-2}}\, ,\label{e:asymptotic1}\\
       \lim_{r\to \infty} \frac{V(r)}{\Vol (B_1(0))}&=\left(\frac{\V_M}{\Vol (B_1(0))}\right)^{\frac{2}{n-2}}\, .\label{e:asymptotic2}
       \end{align}
       \end{Thm}
       
          \begin{proof}  
          By the Bishop-Gromov volume comparison theorem if $r\geq r_0>0$, then
          \begin{align}
          r^{-n}\,\Vol (B_r(x))\leq r_0^{-n}\,\Vol (B_{r_0}(x))\, .
       \end{align}
       Hence, by the Li-Yau, \cite{LY}, lower bound for the Green's function
        \begin{align}
        C\,\int_{d(x,y)}^{\infty} \frac{s}{\Vol (B_{s}(x))}\,ds\leq G(x,y)\ .
        \end{align}
        It follows that if $d(x,y)\geq r_0$, then
        \begin{align}
        G(x,y)\geq \frac{C}{r_0^{-n}\,\Vol (B_{r_0}(x))}\,d^{2-n}(x,y)\, ,
        \end{align}
        and thus by the Cheng-Yau, \cite{CgY}, gradient estimate at such a $y$
       \begin{align}
        |\nabla b|=b\,|\nabla \log b|\leq C\,\frac{G^{\frac{1}{2-n}}}{r} 
        \leq C\, \left[r_0^{-n}\,\Vol (B_{r_0}(x))\right]^{\frac{1}{n-2}}\, .
        \end{align}
        From this the claim follows if $M$ has sub-Euclidean volume growth, i.e. if $\V_M=0$.
        
         Suppose therefore that $M$ has Euclidean volume growth. In this case \eqr{e:asymptotic2} follows from the Bishop-Gromov volume comparison theorem together with (3.38) on page 1374 of \cite{CM2}; cf. also with the proof of Theorem \ref{t:asympsharpGreen} and \cite{ChC1}.   To get \eqr{e:asymptotic1} we argue as follows:
         From Theorem \ref{t:mainmono} and since $V$ is almost constant for $r$ large we have by \eqr{e:asymptotic2} and \cite{CM2} that $A$ is almost constant for $r$ large.    Equation \eqr{e:coareaforV} gives that this constant is the desired one. 
         \end{proof}
        
        It follows easily from Theorem \ref{t:asymptotic} and \eqr{e:diff} that we have the following characterization of Euclidean space as the only manifold with nonnegative Ricci curvature where $V_{\infty}$ is bounded.
               
       \begin{Cor}
       Let $M^n$ be a manifold with nonnegative Ricci curvature, then 
        \begin{align}
       \inf V_{\infty}>-\infty
       \end{align}
       if and only if $M$ is Euclidean space.
       \end{Cor}

            \subsection{The $\cL$ operator and estimates for $b$}
            
            Define a drift Laplacian on the manifold $M$ by
       \begin{align}
       \cL\,u =G^{-2}\dv \left(G^2\,\nabla u\right)=\Delta\,  u+2\, \langle \nabla \log G,\nabla u\rangle\, . 
       \end{align}
       From Lemma \ref{e:LaplE} we get the following useful result:

       \begin{Lem}   \label{l:cL}
      \begin{align}  \label{e:cL}
       \cL\,|\nabla b|^2 &= \frac{1}{2b^2}\,\left(\left|\Hess_{b^2}-\frac{\Delta b^2}{n}g\right|^2+\Ric (\nabla b^2,\nabla b^2)\right)\, ,\\
       \cL\,b^2 &=2\,(4-n)\,|\nabla b|^2\, ,\\
       \cL\,b^{n-2} &=0\, .
       \end{align}
       \end{Lem}

       \begin{proof}
       The first equality follows directly from \eqr{e:bochner3}.  
       The second claim follows from an easy computation using that $\Delta\,b^2=2n\,|\nabla b|^2$ 
       and the last claim follows easily from that $\nabla b^{n-2}=\nabla G^{-1}=-G^{-2}\,\nabla G$ and that $G$ is harmonic.
    \end{proof}

       It follows from this lemma that on a manifold with nonnegative Ricci curvature at a maximum (or minimum) for $|\nabla b|^2$ the hessian of $b^2$ is a multiple of the identity.  
       Since $\Delta b^2=2n\, |\nabla b|^2$ we get that at a maximum
       \begin{align}
       \Hess_{b^2}=2\,|\nabla b|^2\, g\, .
       \end{align}

            The first two inequalities of the next lemma are proven assuming that $G$ is the Green's function whereas the 
       third inequality holds for any positive harmonic function $G$ with $\frac{1}{G}$ proper.
    
       \begin{Lem}   \label{l:cLsub1}
       On a manifold with nonnegative Ricci curvature if $x$ is a fixed point and $r$ is the distance to $x$, then
       \begin{align}  
       b&\leq r\, ,\\
       |\nabla b|&\leq C=C(n)\, ,\\
       0&\leq \cL\,|\nabla b|^2\, .
       \end{align}
        \end{Lem}
        
        \begin{proof}
        The last claim is a direct consequence of the previous lemma.
        
        To see the first and second claim observe first that it follows from the maximum principle together with the Laplace comparison theorem that
     \begin{align}
     r^{2-n}\leq G\, .
     \end{align}
     Therefore on such a manifold 
     \begin{align}
     b\leq r\, .
     \end{align}
     
     To see the second claim observe first that
     \begin{align}
       \nabla \log G=(2-n)\,\nabla \log b\, .
       \end{align}
     Combining this with the Cheng-Yau gradient estimate, \cite{CgY}, applied to the harmonic function $G$ gives that for some constant $C=C(n)$
      \begin{align}
     |\nabla b|\leq \frac{C\,b}{r}\leq C\, .
     \end{align}
     
     The last inequality is an immediate consequence of Lemma \ref{l:cL}.
        \end{proof}

        Recall that for a smooth function $u:M\setminus \{x\}\to \RR$ we set
       \begin{align}
     I_u(r)=r^{1-n}\int_{b=r}u\,|\nabla b|\, .
     \end{align}

        \begin{Lem}  \label{l:cLsub} 
        Let $M^n$ be a manifold and suppose that $u:M\setminus \{x\}\to \RR$ is a smooth function, then  for $r_2>r_1>0$
       \begin{align}
     I_u'(r_2)= r_2^{n-3}\,r_1^{3-n}\,I_u'(r_1)+ r_2^{n-3}\int_{r_1\leq b\leq r_2}\cL\,u\,G^2\, .
     \end{align}
        \end{Lem}
        
        \begin{proof}
        Since for $r_2>r_1>0$
        \begin{align}
        \int_{r_1\leq b\leq r_2}\cL\,u\,G^2= \int_{r_1\leq b\leq r_2}\dv(G^2\,\nabla u)=r_2^{4-2n}\int_{b=r_2}u_n-r_1^{4-2n}\int_{b=r_1}u_n\, .
        \end{align}
        It follows that
        \begin{align}
        \int_{b=r_2}u_n= r_2^{2n-4}\,r_1^{4-2n}\int_{b=r_1}u_n+r_2^{2n-4}\int_{r_1\leq b\leq r_2}\cL\,u\,G^2\,
        \end{align}
        and therefore
        \begin{align}
        I_u'(r_2)= r_2^{n-3}\,r_1^{3-n}\,I_u'(r_1)+ r_2^{n-3}\int_{r_1\leq b\leq r_2}\cL\,u\,G^2\, .
        \end{align}
        \end{proof}
        
        \begin{Cor}  \label{c:cLmono} 
        Let $M^n$ be a manifold with $n\geq 3$ and suppose that $u:M\setminus \{x\}\to \RR$ is a $\cL$-subharmonic function that is bounded from above, then 
       \begin{align}
     I_u(r)=r^{1-n}\int_{b=r}u\,|\nabla b|\, .
     \end{align}
     is monotone nonincreasing.
        \end{Cor}
        
        \begin{proof}
        Since $\cL\,u\geq 0$ it follows from Lemma \ref{l:cLsub} that for $r_2>r_1>0$
        \begin{align}
        I_u'(r_2)\geq r_2^{n-3}\,r_1^{3-n}\,I_u'(r_1)\, .
        \end{align}
        Since $u$ is bounded from above, then $I_u$ is bounded from above and hence, we conclude that 
        \begin{align}
        I_u'\leq 0\, .
        \end{align}
        \end{proof}

        Note that if $u:M\to \RR$ is a smooth function, then 
        \begin{align}
        \lim_{r\to 0}I_u'(r)&=0 \, ,\\
        \lim_{r\to 0}\frac{I_u'(r)}{r}&=\Vol (B_1(0))\,\Delta u(x) \, .
        \end{align}
        
        \begin{Cor}  \label{l:cLmono} 
        Let $M^n$ be a $n$-dimensional manifold and suppose that $u:M\to \RR$ is a smooth function.  If $n=3$, then 
       \begin{align}
     I_u'(r)= \int_{b\leq r}\cL\,u\,G^2\, ,
     \end{align}
     and if $n=4$, then
     \begin{align}
     I_u'(r)= r\, \Vol (\partial B_1(0))\,\Delta u(x)+r\, \int_{b\leq r}\cL\,u\,G^2\, .
     \end{align}
        \end{Cor}

        \subsection{Properties of $A$ and $V$}
        
        \begin{Cor}  \label{c:properties}
        On any manifold $M^n$ with nonnegative Ricci curvature and $n\geq 3$, then $A$, $V$, and $V_{\infty}$ are nonincreasing and bounded from above by what they are on $\RR^n$.  
        Moreover, $A\leq n\,V$.
        \end{Cor}
        
        \begin{proof}
        By Lemma \ref{l:cLsub1} it follows that $|\nabla b|^2$ is bounded and $\cL$-subharmonic, hence, by Corollary \ref{c:cLmono} 
        \begin{align}
        A=I_{|\nabla b|^2}
        \end{align}
        is monotone nonincreasing.  Since $A$ start off what it is in Euclidean space by Lemma \ref{l:ref}, then we get the claim for $A$.  Using Theorem \ref{t:mainmono} we have that
        \begin{align}
        0\geq A'\geq 2(n-1)\,V'
        \end{align}
        this gives the claim for $V$ as $V$ also start off being equal to what it is in Euclidean space by Lemma \ref{l:ref}.  Finally, Theorem \ref{t:secondmono2} now gives that $V_{\infty}$ is nonincreasing.
        
        The last claim follows from that $V'\leq 0$ and $V'=\frac{1}{r}\left(A-n\,V\right)$ by Lemma \ref{l:coarea}.
        \end{proof}

        Combining Lemmas \ref{l:cL} and \ref{l:cLsub} also leads us to our third monotonicity formula:
        
        \begin{Thm}  \label{t:thirdmono} 
        For $r_2>r_1>0$
       \begin{align}
     r^{3-n}_2\,A'(r_2)-r^{3-n}_1\,A'(r_1)=\frac{1}{2}\int_{r_1\leq b\leq r_2}\left(\left|\Hess_{b^2}-\frac{\Delta b^2}{n}g\right|^2+\Ric (\nabla b^2,\nabla b^2)\right)\,b^{2-2n}\, .
     \end{align}
        \end{Thm}

        By Corollary \ref{c:cLmono} in low dimensions we get the following two corollaries:
        
        \begin{Cor}  \label{t:thirdmono}
         Let $M^3$ be a $3$-dimensional manifold.  If $|\nabla b|^2$ is $C^2$ in a neighborhood of $x$, then 
         \begin{align}
     A'(r)= \frac{1}{2}\int_{b\leq r}\left(\left|\Hess_{b^2}-\frac{\Delta b^2}{n}g\right|^2+\Ric (\nabla b^2,\nabla b^2)\right)\,b^{-4}\, ;
     \end{align}
     if in addition $M$ has nonnegative Ricci curvature, then $M$ is flat $\RR^3$.
         \end{Cor}
         
         \begin{proof}
         The first claim follows directly from Corollary \ref{c:cLmono} and the second claim follows from the first together with Corollary \ref{c:properties}.  Namely, combining these it follows that $A$ is constant and hence
         \begin{align}
       \Hess_{b^2}&=\frac{\Delta b^2}{n}\,g\, ,\\
       \Ric (\nabla b^2,\nabla b^2)&=0\, .
        \end{align}
        From this it now follows from section 1 of \cite{ChC1} that $M$ is flat $\RR^3$.
         \end{proof}

     \begin{Cor}
     Let $M^4$ be a $4$-dimensional manifold.  If $|\nabla b|^2$ is $C^2$ in a neighborhood of $x$, then
     \begin{align}
     A'(r)= r\, \Vol (\partial B_1(0))\,\Delta |\nabla b|^2 (x)+\frac{r}{2}\, \int_{b\leq r}\left(\left|\Hess_{b^2}-\frac{\Delta b^2}{n}g\right|^2+\Ric (\nabla b^2,\nabla b^2)\right)\,b^{-6}\, .
     \end{align}
        \end{Cor}

       \section{Sharp gradient estimates for the Green's function}
       
       A natural question is wether the above monotonicity formulas are related to a sharp gradient estimate for the Green's function parallel to that Perelman's monotonicity formula for the Ricci flow is closely related to the sharp gradient estimate of Li-Yau, \cite{LY}, for the heat kernel.
       
       We will see next that the answer to this question is `yes':  
               
         \begin{Thm}  \label{t:sharpGreen}
        If $M^n$ has nonnegative Ricci curvature with $n\geq 3$, then 
        \begin{align}
        |\nabla b|\leq 1\, .
        \end{align}
        Moreover, if equality holds at any point on $M$, then $M$ is flat Euclidean space $\RR^n$.
        \end{Thm}
        
         \begin{proof}
         Given $\epsilon> 0$ by choosing $r> 0$ sufficiently small we have by \eqr{e:ref2} in Lemma \ref{l:ref} that
         \begin{align}
         \sup_{\partial B_r}|\nabla b|^2\leq 1+\epsilon\, .
         \end{align}
         Let $C$ be the gradient bound for $b$ given by Lemma \ref{l:cLsub1} and set 
         \begin{align}
         u=|\nabla b|^2-(1+\epsilon)-C^2\,\frac{b^{n-2}}{R^{n-2}}\, ,
         \end{align}
         then 
         \begin{align}
         \sup_{\partial B_r\cup \{b=R\}}\,u\leq 0\, .
         \end{align}
         From Lemma \ref{l:cL} we have that 
         \begin{align}
         \cL\,u\geq 0\, .
         \end{align}
         By the maximum principle for the operator $\cL$ applied to $u$ we have for $y\in M\setminus \{x\}$ fixed that
         \begin{align}
         |\nabla b|^2(y)\leq 1+\epsilon+C^2\,\frac{b^{n-2}(y)}{R^{n-2}}\, .
         \end{align}
         Letting $\epsilon\to 0$ and $R\to \infty$ gives the inequality.
         
         To prove that Euclidean space is characterized by that equality holds; suppose that at some point $p\in M$ we have that $|\nabla b|^2(p)=1$.   Since $|\nabla b|^2\leq 1$, 
         $\cL\,|\nabla b|^2\geq 0$, and $p$ is an interior point in $M\setminus \{x\}$ where the maximum of $|\nabla b|^2$ is achieved it follows from the maximum principle that $|\nabla b|^2\equiv 1$ everywhere and thus by \eqr{e:cL} 
         \begin{align}
       \Hess_{b^2}&=\frac{\Delta b^2}{n}\,g\, ,\\
       \Ric (\nabla b^2,\nabla b^2)&=0\, .
        \end{align}
        From this it follows from section 1 of \cite{ChC1} that $M$ is a metric cone and that $b$ is the distance to the vertex.  Since Euclidean space is the only smooth cone the claim follows.
         \end{proof}
         
         We next give a slightly different proof of Theorem \ref{t:sharpGreen} that instead of using the $\cL$ operator use that  $|\nabla b|^2\,G$ is subharmonic by \eqr{e:bsub}.
         
         \begin{proof}
         (Alternate proof of the sharp bound in Theorem \ref{t:sharpGreen}).   
         Given $\epsilon> 0$ by choosing $r> 0$ sufficiently small and $R$ sufficiently we have by \eqr{e:ref2} in Lemma \ref{l:ref} and since $G\to 0$ at infinity that
         \begin{align}
         \sup_{\partial B_r}|\nabla b|^2\leq 1+\epsilon\, ,
         \end{align}
         and
         \begin{align}
         \sup_{\partial B_R}G\leq \epsilon\, .
         \end{align}
         Let $C$ be the gradient bound for $b$ given by Lemma \ref{l:cLsub1} and set 
         \begin{align}
         u=|\nabla b|^2\,G-(1+\epsilon)\,G-C^2\,\epsilon\, ,
         \end{align}
         then 
         \begin{align}
         \sup_{\partial B_r\cup \{b=R\}}\,u\leq 0\, .
         \end{align}
         By \eqr{e:bsub} we have that 
         \begin{align}
         \Delta\,u\geq 0\, .
         \end{align}
         By the maximum principle for the Laplacian applied to $u$ we have for $y\in M\setminus \{x\}$ fixed that
         \begin{align}
         \left[|\nabla b|^2(y)-(1+\epsilon)\right]\,G(y)\leq C^2\,\epsilon\, .
         \end{align}
         Letting $\epsilon\to 0$ gives the inequality.
         \end{proof}

         The argument in the proof of Theorem \ref{t:sharpGreen} in fact gives that the $\sup_{b= r}|\nabla b|^2$ is monotone nonincreasing in $r$ or slightly more general:
         
         \begin{Thm}  \label{t:supnoninc}
         Let $\Omega$ be open bounded subset of $M$ containing $x$,  then for all $y\in M\setminus \overline{\Omega}$
         \begin{align}
         |\nabla b|^2(y)\leq \sup_{\partial\Omega}|\nabla b|^2\, .
        \end{align}
        Moreover, strict inequality holds unless $M$ is isometric to a cone outside a compact set.  
         \end{Thm}
         
         Theorem \ref{t:supnoninc} should be compared with that $A$ is monotone nonincreasing by Corollary \ref{c:properties}.

         \vskip2mm
         In terms of $G$ this sharp gradient estimate is the following:
         
         \begin{Cor}
         If $M^n$ has nonnegative Ricci curvature with $n\geq 3$, then 
          \begin{align}
        |\nabla G|\leq (n-2)\,G^{\frac{n-1}{n-2}}\, .
        \end{align}
         \end{Cor}
        
         \begin{proof}
          \begin{align}
        |\nabla \log G|= (n-2)\,|\nabla \log b|\leq (n-2)\, \frac{|\nabla b|}{b}\leq (n-2)\, G^{\frac{1}{n-2}}\, .
        \end{align}
         \end{proof}
         
         Another immediate corollary of the sharp gradient estimate for the Green's function is the following:
         
         \begin{Cor}
         If $M^n$ has nonnegative Ricci curvature with $n\geq 3$, then for all $r> 0$
          \begin{align}
        \Vol (b=r)&\geq \Vol (\partial B_r(0))\, ,\\
        \Vol (b\leq r)&\geq \int_0^r\Vol (b=s)\geq \Vol (B_r(0))\, .
        \end{align}
         \end{Cor}
         
         \begin{proof}
         To see the first claim note that by the sharp gradient estimate 
          \begin{align}
       \Vol (b= r)\geq \int_{b\leq r}|\nabla b|=r^{n-1}\,\Vol (B_r(0))\, .
        \end{align}
        The second claim follows from the coarea formula and the sharp gradient estimate.   Namely, by those two we have that
        \begin{align}
       \Vol (b\leq r)= \int_{b\leq r}\frac{|\nabla b|}{|\nabla b|}=\int_0^r\int_{b=s}\frac{1}{|\nabla b|}\geq \int_0^r\Vol (b=s)\, .
        \end{align}
        Here the last inequality used the first claim.
         \end{proof}

                We show next a sharp asymptotic gradient estimate for the Green's function on manifolds with nonnegative Ricci curvature:

         \begin{Thm}  \label{t:asympsharpGreen}
        If $M^n$ has nonnegative Ricci curvature with $n\geq 3$, then 
        \begin{align}
        \lim_{r\to \infty} \sup_{M\setminus B_r(x)}|\nabla b|=\left(\frac{\V_M}{\Vol (B_1(0))}\right)^{\frac{1}{n-2}}\, .
        \end{align}
        \end{Thm}
        
        To prove this theorem we will need the following lemma that was proven in \cite{ChCM}); see the proof of (\#) on page 952 of \cite{ChCM}).  (For completeness and since this was not explicitly  stated as a lemma there we will include the proof).
        
        \begin{Lem}  \label{l:ccm}
        ((\#) on page 952 of \cite{ChCM}).  
        Let $M^n$ be an open manifold with nonnegative Ricci curvature and let $u$ be a positive superharmonic function on $B_r(x)$.  Then there exists a constant $C=C(n)$ such that
        \begin{align}
        \frac{1}{\Vol (B_{r}(x))}\int_{\partial B_r(x)}u\leq C\,u(x)\, .
        \end{align}
        \end{Lem}
        
        \begin{proof}
        (The proof is taken from \cite{ChCM}).  
        Let $h_r$ be the harmonic function on $B_r(x)$ with $h_r|\partial B_r(x)=u|\partial B_r(x)$.  By the maximum principle $0<h_r\leq u$ so $0<h_r(x)\leq u(x)$.  By the Cheng-Yau Harnack inequality for some $C=C(n)$ 
        \begin{align}
        \sup_{B_{\frac{r}{2}}(x)} h_r\leq C\,\inf_{B_{\frac{r}{2}}(x)}h_r\leq C\,u(x)\, .
        \end{align}
        Moreover, by the Laplacian comparison theorem and since $h_r$ is harmonic and nonnegative
        \begin{align}
        \log \left(s^{1-n}\int_{\partial B_s(x)}h_r\right)\downarrow 
        \end{align}
        is monotone nonincreasing.    Combining this gives that
        \begin{align}
        r^{1-n}\int_{\partial B_{r}(x)}h_r\leq \left(\frac{r}{2}\right)^{1-n}\int_{\partial B_{\frac{r}{2}}(x)}h_r\leq C\,\left(\frac{r}{2}\right)^{1-n}\,\Vol (\partial B_{\frac{r}{2}}(x))\, u(x)\, .
        \end{align}
        Hence,
        \begin{align}
        \frac{1}{\Vol (\partial B_{r}(x))}\int_{\partial B_r(x)}u= \frac{1}{\Vol (\partial B_{r}(x))}\int_{\partial B_r(x)}h_r\leq 2^{n-1}\,C\,\frac{\Vol (\partial B_{\frac{r}{2}}(x))}{\Vol (\partial B_{r}(x))}\,u(x)
        \leq 2^{n}\,C\,n\,u(x)\, .
        \end{align}
        \end{proof}

           \begin{proof}
           (of Theorem \ref{t:asympsharpGreen}). 
           It follows from the proof of  Theorem \ref{t:asymptotic} that we only need to show the theorem when $M$ has Euclidean volume growth.
           
           Set 
        \begin{align}
        L=\sup_{M\setminus B_r(x)}|\nabla b|^2\, ,
        \end{align}
         and let $y\in M\setminus B_{2r}(x)$.
        It follows from the Cheng-Yau Harnack inequality for $G$ it follows that
        \begin{align}
        \sup_{B_{\frac{r}{2}}(y)} G\leq C \inf_{B_{\frac{r}{2}}(y)}G\, .
        \end{align}
        Combining this with Lemma \ref{l:ccm} applied to $G\,(L-|\nabla b|^2)$ since 
          \begin{align}
        \Delta\,G\,(L-|\nabla b|^2)=- \Delta\,G\,|\nabla b|^2\leq 0\, ,
        \end{align}
        we get that
        \begin{align}
        \frac{1}{C\,\Vol (\partial B_{r}(y))}\int_{\partial B_{r}(y)}(L-|\nabla b|^2) 
        \leq \frac{1}{\Vol (\partial B_{r}(y))}\int_{\partial B_{r}(y)}G\,(L-|\nabla b|^2) \leq C\,(L-|\nabla b|^2)(y) \, .
        \end{align}
        All we need to show is therefore that the average of $|\nabla b|^2$ on all balls of radius $r$ centered at $\partial B_{2r}(x)$ converges to 
         \begin{align}
        \left(\frac{\V_M}{\Vol (B_1(0))}\right)^{\frac{2}{n-2}}
        \end{align}
        as $r\to \infty$.   This however follows from the Bishop-Gromov volume comparison theorem together with (3.38) on page 1374 of \cite{CM2}; cf. also with the proof of Theorem \ref{t:asymptotic}.  
         \end{proof}   
         
         For the Green's function itself this sharp asymptotic gradient estimate is:
         
         \begin{Cor}  \label{c:asympsharpGreen}
        If $M^n$ has nonnegative Ricci curvature with $n\geq 3$, then 
        \begin{align}
        (n-2)\lim_{r\to \infty} \sup_{M\setminus B_r(x)}\frac{|\nabla G|}{G^{\frac{n-1}{n-2}}}=\left(\frac{\V_M}{\Vol (B_1(0))}\right)^{\frac{1}{n-2}}\, .
        \end{align}
        \end{Cor}
        
        Even on an open manifold with Euclidean volume growth and nonnegative Ricci curvature, where by the above theorem $\sup_{M\setminus B_r(x)} |\nabla b|$ converges to its nonzero average, $\nabla b$ may vanishes arbitrarily far out.  Indeed, X. Menguy, \cite{M1}, has given examples of such manifolds with infinite topological type and thus $\nabla b$ in each of those examples vanish arbitrarily far out; cf also with \cite{P3}.  For Ricci-flat manifolds with Euclidean volume growth the corresponding question is unknown; without the assumption of Euclidean volume growth there are examples of Anderson-Kronheimer-LeBrun, \cite{AKL}, of Ricci-flat manifolds with infinite topological type so in those examples $\nabla b$ vanishes arbitrarily far out too.
          
               \section{Distance to the space of cones and uniqueness}
       
       In this section we will relate the derivative of the monotone quantities from the previous section to the distance to the nearest cone.  Using this we get a uniqueness criteria for tangent cones.

\subsection{Distance to the space of cones and a uniqueness criteria}  Recall that a metric cone $C(Y)$ over a metric space $(Y,d_Y)$ is the metric completion of the set $(0,\infty)\times Y$ with the metric
\begin{align}
d^2_{C(Y)}((r_1,y_1),(r_2,y_2))=r_1^2+r_2^2-2\,r_1\,r_2\,\cos d_Y(y_1,y_2)\, ;
\end{align}
see also section 1 of \cite{ChC1}.  When $Y$ itself is a complete metric space taking the completion of $(0,\infty)\times Y$  adds only one point to the space.  This one point is usually referred to as the vertex of the cone.  We will also sometimes write $(0,\infty)\times_r Y$ for the metric cone.

\vskip2mm
We will next define a scale invariant notion that measure how far the metric space on a given scale is from a cone.  This is the following:

\begin{Def}
(Scale invariant distance to the space of cones.)  
Suppose that $(X,d_X)$ is a metric space and $B_r(x)$ is a ball in $X$.  Let $\Theta_r(x)>0$ be the infimum of all $\Theta>0$ such that 
\begin{equation}
d_{GH}(B_r(x),B_r(v))<\Theta\,r\, ,
\end{equation}
where $B_r(v)\subset C(Y)$ and $v$ is the vertex of the cone.
\end{Def}

For the discussion that follows it is useful to keep the following example in mind:

\begin{Exa}
(Koch curve.)  Let $K_1$ be the union of two line segments of length $1$ meeting at an angle of almost $\pi$.  Replace each of the two line segments by a scaled down copy of $K_1$ to get $K_2$.  Repeat this process and denote the i'th curve by $K_i$.  The Koch curve is the limit as $i\to \infty$.

The Koch curve is an example of a set that is not bi-Lipschitz to a line yet for all $r>0$ it satisfies
\begin{equation}
\Theta_r<\epsilon\, ,
\end{equation}
with $\epsilon\to 0$ as the angle in $K_1$ tend to $\pi$.  
Tangent cones for the Koch curve are not unique.    The Koch curve is obviously a metric space with the metric induced from $\RR^2$ however it is not a length space as it is a curve with Hausdorff dimension $>1$.
\end{Exa}

\subsection{Criteria for uniqueness} 

We have that the following integrability that implies uniqueness:

\begin{Thm}   \label{t:criteria1a} 
If $\alpha>1$ and 
\begin{equation}
\int_1^{\infty} \frac{\Theta_r^2}{r\,|\log r|^{-\alpha}}\,dr<\infty\, ,
\end{equation}
then the tangent cone at infinity is unique.  Likewise for tangent cones at a point.
\end{Thm}

\begin{proof}
By the Cauchy-Schwarz inequality and some elementary inequalities for $R$ sufficiently large
\begin{equation}
\left(\sum_k\Theta_{\e^{-k}}\right)^2\leq \sum_k\Theta_{\e^{-k}}^2\,k^{\alpha}\,\sum_k k^{-\alpha}\leq \int_R^{\infty} \frac{\Theta_r^2}{r\,|\log r|^{-\alpha}}\,dr\,\sum_k k^{-\alpha}<C^2\,\epsilon^2\, .
\end{equation}
Hence,  by the triangle inequality
\begin{align}
d_{GH}\left(\frac{1}{r} B_r(x),\frac{1}{\e r}B_{\frac{1}{\e r}}(x)\right)&\leq \frac{1}{r}d_{GH}(B_r(x),B_r(v_{\e r}))+\frac{1}{\e r} d_{GH}(B_{\e r}(x),B_{\e r}(v_{\e r}))\\&\leq \e\, \Theta_{\e r}+\Theta_{\e r}=(\e +1)\,\Theta_{\e r}\, .\notag
\end{align}
Therefore if we set $r_k=\e^k$, then
\begin{align}
d_{GH}\left(\frac{1}{r_m}B_{r_m}(x),\frac{1}{r_{1}}B_{r_{1}}(x)\right)&\leq \sum_{k=1}^{m-1}d_{GH}\left(\frac{1}{r_k}B_{r_k}(x),\frac{1}{r_{k+1}}B_{r_{k+1}}(x)\right)\\
&\leq (1+\e)\,\sum_k\Theta_{r_{k+1}}\leq (1+\e)\,C\,\epsilon\, .\notag
\end{align}
\end{proof}

Another closely related criterium for uniqueness is the following: 

\begin{Thm}   \label{t:criteria1} 
 If $F$ is a nonnegative function on $[1,\infty)$ with $-F'\geq F^{1+\alpha}$ for some $\alpha> 0$ and $-C\,F'( s)\geq \Theta_{\e^s}^{2+2\epsilon}$ for some constant $C$ where 
 $\frac{1}{\alpha}-1>2\epsilon \geq 0$, 
then the tangent cone at infinity is unique.  Likewise for tangent cones at a point.
\end{Thm}

This theorem will be an immediate consequence of the next lemma, its corollary, and the triangle inequality; where the triangle inequality 
is applied as in the proof of Theorem \ref{t:criteria1a}.
  
  \begin{Lem}   \label{l:intfinite}
  If $F$ is a nonnegative function on $[1,\infty)$ with $-F'\geq F^{1+\alpha}$ for some $\alpha> 0$, 
  then for $\frac{1}{\alpha}-1>\beta\geq 0$
  \begin{align}
  \int_0^{\infty}F'\,r^{1+\beta}>-\infty\, .
  \end{align}
\end{Lem}

\begin{proof}
From the assumption it follows that
\begin{align}
\left(F^{-\alpha}\right)'\geq \alpha\, .
\end{align}
Hence, for $0\leq s\leq t$
\begin{align}
F^{-\alpha}(t)-F^{-\alpha}(s)\geq \alpha\, (t-s)\, .
\end{align}
Therefore, 
\begin{align}
F(t)\leq \left(\alpha\, (t-s)+F^{-\alpha}(s)\right)^{-\frac{1}{\alpha}}\leq \left(\alpha\, (t-s)\right)^{-\frac{1}{\alpha}}\, .
\end{align}

We can now bound the integral as follows
\begin{align}
-\int_1^{\infty}F'\, r^{1+\beta}&=-\sum_j\int_{2^j}^{2^{j+1}}F'\, r^{1+\beta}\leq -\sum_j2^{(j+1)(1+\beta)}\int_{2^j}^{2^{j+1}}F'\notag\\
&\leq \sum_j2^{(j+1)(1+\beta)}\left(F(2^j)-F(2^{j+1})\right)\leq \sum_j2^{(j+1)(1+\beta)}F(2^j)\\
&\leq \sum_j2^{(j+1)(1+\beta)}\left(\alpha 2^j\right)^{-\frac{1}{\alpha}}\leq \alpha^{-\frac{1}{\alpha}}\,2^{1+\beta}\,\sum_j2^{j(1+\beta-\frac{1}{\alpha})}\, .\notag
\end{align}
The claim follows since this sum is finite when  $\frac{1}{\alpha}-1>\beta$.
 \end{proof}
 
 \begin{Cor}
If $F$ and $\beta$ are as in Lemma \ref{l:intfinite} and $\Theta$ is a nonnegative function on $[1,\infty)$ with $-C\,F'(s)\geq \Theta_{\e^s}^{2+2\epsilon}$ for some constant $C$ where $\beta\geq 2\epsilon\geq 0$, 
then 
  \begin{align}
  \int_{\e}^{\infty}\frac{\Theta_{r}}{r}\, dr =\int_1^{\infty}\Theta_{\e^s}\, ds <\infty\, .
  \end{align}
 \end{Cor}
 
 \begin{proof}
 By the assumption and Lemma \ref{l:intfinite}
 \begin{align}
 \int_1^{\infty} \Theta_{\e^s}^{2+2\epsilon}\,s^{1+2\epsilon}\, ds\leq -C\int_1^{\infty} F'(s)\,s^{1+\beta}\, ds<\infty\,  .
  \end{align}
  Combining this with the Cauchy-Schwarz inequality gives that
   \begin{align}
 \int_1^{\infty} \Theta_{\e^s}\, ds\leq \left(\int_1^{\infty} \Theta_{\e^s}^{2+2\epsilon}\,s^{1+2\epsilon}\, ds\right)^{\frac{1}{2}}\,\left(\int_1^{\infty} s^{-1-2\epsilon}\, ds\right)^{\frac{1}{2}}<\infty\,  .
  \end{align}
 \end{proof}
 
\subsection{Bounding the distance to cones} 

The following way of bounding the distance to the space of cones will be key later on:

\begin{Thm}  \label{t:fund}
Given $\epsilon$, $\V_m>0$, there exist $C(\epsilon,n,\V_m)>0$ and $c=c(n,\V_m)>1$ such that the following holds:  Let $M$ be an $n$-dimensional manifold with nonnegative Ricci curvature and Euclidean volume growth.  If $\V_M\geq \V_m$ and $b=G^{\frac{1} {2-n} } $, where $G=G(x,\cdot)$ is the Green's and $x\in M$ is fixed, then for $r$ sufficiently large
\begin{equation}
\Theta_{\frac{r}{c}}^{1+\epsilon} \leq C\, r^{-n} \int_{b\leq r} \left| \Hess_{b^2}-\frac{\Delta b^2}{n}g\right|\, .
\end{equation}
\end{Thm}

The proof of this theorem uses \cite{ChC1} and will be given in \cite{CM4}.

\begin{Cor}  \label{c:fund}
Given $\epsilon$, $\V_m>0$, there exist $C(\epsilon,n,\V_m)>0$  and $c=c(n,\V_m)>1$ such that the following holds:  Let $M$ be an $n$-dimensional manifold with nonnegative Ricci curvature and Euclidean volume growth.  If $\V_M\geq \V_m$ and $b=G^{\frac{1} {2-n} } $, where $G=G(x,\cdot)$ is the Green's and $x\in M$ is fixed, then for $r$ sufficiently large
\begin{equation}
\Theta_{\frac{r}{c}}^{2+2\epsilon} \leq C\, r^{-n} \int_{b\leq r} \left| \Hess_{b^2}-\frac{\Delta b^2}{n}g\right|^2\, .
\end{equation}
\end{Cor}

\begin{proof}
This follows from Theorem \ref{t:fund} by the Cauchy-Schwarz inequality.
\end{proof}

Combining Theorem \ref{t:mainmono} with Corollary \ref{c:fund} we get the following inequality (see \cite{CM4} for more details):

\begin{Thm}  \label{t:fund2a}
Given $\epsilon>0$, there exist $C=C(\epsilon,n,\V_M)>0$  and $c=c(n,\V_m)>1$ such that for $r$ sufficiently large
\begin{align}
      C\, \left(A-2(n-1)\, V\right)'\geq  \frac{\Theta_{\frac{r}{c}}^{2+2\epsilon}}{r}\, .
       \end{align}
\end{Thm}

Likewise from Theorem \ref{t:secondmono1} and Corollary \ref{c:fund} we get the following inequality (see \cite{CM4} for more details):

\begin{Thm}  \label{t:fund3a}
Given $\epsilon>0$, there exist $C=C(\epsilon,n,\V_M)>0$  and $c=c(n,\V_m)>1$ such that for $r$ sufficiently large
\begin{equation}
-C\,\left(r^{2-n} \, \left[A-\Vol (\partial B_1(0))\right]\right)'\geq r^{1-n}\int_0^r\frac{\Theta_{\frac{s}{c}}^{2+2\epsilon}}{s}\,ds\, .
\end{equation}
\end{Thm}

\subsection{Uniqueness criteria revisited}

By combining a number of the results above we can now show the following uniqueness criteria for manifolds with nonnegative Ricci curvature and Euclidean volume growth:

\begin{Thm}  \label{t:uniqueC}
Let $M$ be an $n$-dimensional manifold with nonnegative Ricci curvature and Euclidean volume growth.  If for constants $\alpha> 0$, $K$, and all $r$ sufficiently large
\begin{align}
2(n-1)\, V-A&\geq K\, ,\\
-r\, \left(2(n-1)\, V-A\right)'&\geq \left(2(n-1)\, V-A-K\right)^{1+\alpha}\, ,
\end{align}
then the tangent cone at infinity of $M$ is unique.
\end{Thm}

\begin{proof}
Set 
\begin{equation}
F_0(r)=2(n-1)\, V(c r)-A(c r)-K\, ,
\end{equation}
and 
\begin{equation}
F(s)=F_0(\e^s)\, .
\end{equation}
Then $F$ is nonnegative, 
\begin{equation}
-F'(s)=-\e^s\,F_0'(\e^s)=c\, \e^s\, \left(A-2(n-1)\, V\right)'(\e^s)\geq F^{1+\alpha}(s)\, ,
\end{equation}
and by Theorem \ref{t:fund2a} for $s$ sufficiently large
\begin{equation}
-C\,F'(s)=-C\,\e^s\,F_0'(\e^s)=c\, C\,\e^s\, \left(A-2(n-1)\, V\right)'(c\, \e^s)\geq \Theta_{\e^s}^{2+2\epsilon}\, .
\end{equation}
Uniqueness follows now from Theorem \ref{t:criteria1}.
\end{proof}

\subsection{Dini conditions}

The notion of a set being scale invariantly close to a cone is parallel to the classical Reifenberg condition for $n$ dimensional subset of some big Euclidean space.  Here being close to a cone is replaced by the stronger condition to being close to an $n$-dimensional affine linear subset and Gromov-Hausdorff distance is replaced by Hausdorff distance; see \cite{R}, \cite{T} and compare with the appendix 1 of \cite{ChC1} where this is generalized to metric spaces.   For locally compact Reifenberg sets that satisfies an additional Dini condition, which is very much in the spirit of the earlier discussion of this section, T. Toro has proven that they can be parametrized by maps with bi-Lipschitz constants close to $1$.  

The Dini condition of Toro, \cite{T}, is the condition that for all $x$ 
\begin{equation}
\int_0^1 \frac{ \Theta_{r}^2}{r}\, dr\leq \epsilon^2\, ,
\end{equation}
for some $\epsilon>0$ sufficiently small.  
This can also be expressed as a condition of sum of $\Theta_r^2$ over dyadic small scales.  Namely, as
\begin{equation}
\sum \Theta_{r}^2\leq \epsilon^2\, ,
\end{equation}
for a possibly different $\epsilon>0$.

Note that it follows from our results above that in our setting we have Dini conditions like those of Toro with a power slightly bigger than two and with our $\Theta_r$.  However, without an additional rate of convergence, like that in Theorem \ref{t:criteria1a}, uniqueness of tangent cones does not hold; \cite{P2}, \cite{ChC1}, \cite{CN2}.

  \section{Monotone quantities for heat flow}  \label{s:perelman}

For completeness and for the readers convenience we have included the present section that discuss Perelman's $F$ and $W$ functional in the present setting of manifolds with nonnegative Ricci curvature.    There are very few new things in this section and most of the results can be found in \cite{P1} or \cite{N1}--\cite{N3}; however the presentation we give emphasize the parallels to the previous sections which is also the rational for including it.

\subsection{The quantities}
Let $H(x,y,t)$ be the heat kernel on $M$.  For $x$ fixed set $H_x(y,t)=H(x,y,t)$.  
\vskip2mm
We define a function (the Shannon entropy, cf. \cite{N1}--\cite{N3}) $S$ by
\begin{equation}
S=S_x(t)=-\int_M \log H_x\,H_x-\frac{n}{2}\,\log (4\pi\, t)-\frac{n}{2}=\int_M h\,H_x-\frac{n}{2}\,\log (4\pi\,t)-\frac{n}{2}\, .
\end{equation}
Where $h=-\log H_x$.    The constant $\frac{n}{2}$ comes from that
\begin{equation}
\frac{n}{2}=\int_{\RR^n}\frac{|y|^2}{4}\,\e^{-\frac{|y|^2}{4}}=\int_{\RR^n}\frac{|y|^2}{4t}\,\e^{-\frac{|y|^2}{4t}}\, .
\end{equation} 
Moreover, on Euclidean space $S$ vanishes.    Observe also that on a smooth manifold
\begin{equation}  \label{e:euclheatasymp} 
\lim_{t\to 0}S(t)=0\, .
\end{equation}

Before we introduce the next quantity (which is essentially {Perelman's $\cF$-functional, \cite{P1}) we will need to recall the parabolic gradient estimate.  Namely, Li-Yau, \cite{LY}, showed that for any positive solution $u$ to the heat equation on a manifold with nonnegative Ricci curvature
\begin{align}
-\Delta\,\log u=\frac{|\nabla u|^2}{u^2}-\frac{u_t}{u}\leq \frac{n}{2t}\, .
\end{align}
Note that this quantity vanishes precisely on cones.  Integrating the Li-Yau inequality yields
\begin{equation}
F=F_x(t)=t\int_M\Delta\,h\,H_x-\frac{n}{2}=-t\int_M\langle\nabla h,\nabla H_x\rangle-\frac{n}{2}=t \int_M |\nabla h|^2\,H_x-\frac{n}{2}\leq 0\, .
\end{equation}
We will see shortly that $F=t\,S'$.    $F$ is Perelman's $\cF$-functional adapted to the current setting; see \cite{P1} and cf. \cite{N1}--\cite{N3}.  

Observe that $F$ vanishes on any cone with vertex $x$.  
This is not the case for $S$ rather the value of $S$ depends on the volume growth.  Note also that when $M$ is a smooth manifold, then 
\begin{equation}
\lim_{t\to 0} F(t)=0\, .
\end{equation}
Moreover, when $M$ has nonnegative Ricci curvature and Euclidean volume growth, then it follows easily from the cone structure at infinity, \cite{ChC1}, that 
\begin{equation}  \label{e:Fasymp}
\lim_{t\to \infty}F(t)=0\, ,
\end{equation}
cf. \cite{N2}.  

Perelman, \cite{P1} (see also \cite{N1}--\cite{N3}), defined a closely related functional $W$ by  
\begin{equation}
W(t)=W_x(t)=F(t)+S(t)=\int_M\left(t\,|\nabla f|^2+f-n\right)\,H_x\, .
\end{equation}
Here $f=-\log H_x-\frac{n}{2}\,\log (4\pi t)$.  Moreover, in \cite{P1}, \cite{N1}--\cite{N3}, it is shown that
\begin{equation}  \label{e:deriW}
\frac{d}{dt}W= -2t\int_M\left( \left|\Hess_f-\frac{1}{2t}g\right|^2+\Ric (\nabla f,\nabla f)\right)\,H_x\, .
\end{equation}

\subsection{Entropy calculations and monotonicity}

We begin this section by showing that the derivative of $S$ is given in terms of $F$.  Once we have that it follows immediately what the derivative of $F$ is by the results of Perelman, \cite{P1}, cf. \cite{N1}--\cite{N3}.  

\begin{Lem}  \label{l:deriS}
\begin{align}
F=t\,S'\, .
\end{align}
\end{Lem}

\begin{proof}
A straightforward calculation yields
\begin{align}
S'&=-\partial_t \int_M \log H_x\,H_x-\frac{n}{2t}=-\int_M \partial_tH_x-\int_M\log H_x\,\partial_t H_x-\frac{n}{2t}\\
&=-\partial_t \int_M H_x-\int_M\log H_x\,\Delta H_x-\frac{n}{2t}=\int_M\frac{|\nabla H_x|^2}{H_x}+\frac{n}{2t}
=\int_M |\nabla h|^2\,H_x-\frac{n}{2t}=\frac{F}{t}\leq 0\, .\notag
\end{align}
\end{proof}

Since $W=F+S$ we get the next lemma by combining \eqr{e:deriW} and Lemma \ref{l:deriS}.

\begin{Lem}  \label{l:mono}
\begin{equation}
\partial_t(t\,F)= -2t^2\int_M\left( \left|\Hess_h-\frac{1}{2t}g\right|^2+\Ric (\nabla h,\nabla h)\right)\,H_x
\end{equation}
\end{Lem}

Using the Cauchy-Schwarz inequality twice we get the following:

\begin{Cor}
\begin{equation}
-\partial_t(t\,F)\geq \frac{2}{n}\,F^2\, .
\end{equation}
\end{Cor}

\begin{proof}
By the Cauchy-Schwarz inequality
\begin{align}
\frac{1}{n}\,\left(\Delta\,h-\frac{n}{2t}\right)^2\leq \left|\Hess_h-\frac{1}{2t}\,g\right|^2\, .
\end{align}
Applying the Cauchy-Schwarz inequality one more time yields
\begin{align}
-\partial_t(t\,F)&= 2t^2\int_M\left( \left|\Hess_h-\frac{1}{2t}g\right|^2+\Ric (\nabla h,\nabla h)\right)\,H_x\notag\\
&\geq 2t^2\int_M\left|\Hess_h-\frac{1}{2t}g\right|^2\,H_x\geq \frac{2}{n}\,t^2\left(\int_M\Delta\,h-\frac{n}{2t}\right)^2\,H_x=\frac{2}{n}\,F^2\, .
\end{align}
\end{proof}

For completeness we include the calculation that yields Lemma \ref{l:mono} and thus also \eqr{e:deriW}.  

\begin{proof}
(of Lemma \ref{l:mono}).  
Let $u$ be a positive solution to the heat equation and set $h = -\log u$.  Then
\begin{equation}  \label{e:nice}
(\partial_t-\Delta)\,h=-|\nabla h|^2\, ,
\end{equation}
and 
\begin{equation}
(\partial_t-\Delta)\,\Delta\,h=\Delta (\partial_t-\Delta)\,h=-\Delta |\nabla h|^2\, .
\end{equation}
By \eqr{e:nice} $(\partial_t-\Delta)\,h=-|\nabla h|^2$.  Combining this with the Bochner formula yields
\begin{align}
	\frac{1}{2} \, \left( \partial_t - \Delta \right) \, |\nabla h|^2&= \langle \nabla \partial_th,\nabla h\rangle-\frac{1}{2} \,  \Delta \, |\nabla h|^2\notag\\
	&=  \langle \nabla (\partial_t-\Delta)\,h , \nabla h  \rangle  -  \left| \Hess_h \right|^2  - \Ric (\nabla h , \nabla h)\\
	&=- \langle \nabla h , \nabla |\nabla h|^2  \rangle  -
	\left| \Hess_h \right|^2  - \Ric (\nabla h , \nabla h) 
	\, .\notag
\end{align}
Thus, since $(\partial_t - \Delta) u = 0$ and $u \nabla h =- \nabla u$,
the product rule   gives
\begin{align}	\label{e:pr}
	\frac{1}{2} \, (\partial_t - \Delta)  (u \,  |\nabla h|^2)  &=
	 \langle \nabla u , \nabla    |\nabla h|^2 \rangle - u \, \left| \Hess_h \right|^2 - u \, \Ric (\nabla h , \nabla h)  
	 - \langle \nabla u , \nabla |\nabla h|^2 \rangle\notag    \\
	  &= - u \, \left| \Hess_h \right|^2 - u \, \Ric (\nabla h , \nabla h)  \, .
\end{align}
Differentiating and using Stokes' theorem to get that $\int \Delta (u \, |\nabla h|^2)=0$ gives
\begin{align}
	\partial_t
	\int_M  |\nabla h |^2 \, u   =  \int_M (\partial_t - \Delta)  (u \,  |\nabla h|^2)  
		   = - 2 \int_M u \,\left(  \left| \Hess_{h} \right|^2 + \Ric (\nabla h , \nabla h)  \right) \, ,
\end{align}
where the last equality used \eqr{e:pr}.  Rewriting we get 
\begin{align}
	\partial_t
	\int_M  |\nabla h |^2 \, H_x    &=	- 2 \int_M
	H_x \left( \left| \Hess_{h} \right|^2 + \Ric (\nabla h , \nabla h)
	\right)\notag \\
	&= - 2 \int_M
	H_x \left( \left| \Hess_{h}-\frac{1}{2t}g \right|^2 + \Ric (\nabla h , \nabla h)
	\right)+\frac{2}{t}\int_M\Delta h\,H_x+\frac{n}{2t^2}\int_MH_x\\
          &= - 2 \int_M
	H_x \left( \left| \Hess_{h}-\frac{1}{2t}g \right|^2 + \Ric (\nabla h , \nabla h)\right)-\frac{2}{t}\int_M|\nabla h|^2\,H_x+\frac{n}{2t^2}\, .  \notag
\end{align}
Therefore
\begin{align}
\left(t^2\,\int_M|\nabla h|^2\,H_x\right)'=-2\,t^2\int_M
	H_x \left( \left| \Hess_{h}-\frac{1}{2t}g \right|^2 + \Ric (\nabla h , \nabla h)
	\right)+\frac{n}{2}\, ,
\end{align}
The claim easily follows from this.  
\end{proof}

Using that the derivative of $S$ is given in terms of $F$ we get the following:

\begin{Cor}
Set $J(t)=t\,S(t)$, then
\begin{align}
J'&=W\, ,\\
J''&= -2t\int_M\left( \left|\Hess_f-\frac{1}{2t}g\right|^2+\Ric (\nabla f,\nabla f)\right)\,H\, .
\end{align}
\end{Cor}

As noted earlier (see the discussion surrounding \eqr{e:euclheatasymp}), then it follows easily that for manifolds with nonnegative Ricci curvature and Euclidean volume growth $F(t)\to 0$ as $t\to \infty$.   On the other hand by the Li-Yau gradient estimate the integrand in $F$ is pointwise nonnegative thus the sup of the integrand tends to its average at infinity.  This easy fact for the heat kernel parallels the more complicated sharp asymptotic gradient estimate for the Green's function in Theorem \ref{t:asympsharpGreen}.  
\section{Parabolic distances to cones}

\subsection{Weighted distances to cones} 

We can also define the weighted distance to the space of cones as follows:

\begin{Def}
(Weighted scale invariant distance to the space of cones.)  
Suppose that $M$ is a smooth manifold with nonnegative Ricci curvature and $x\in M$ and $H$ is the heat kernel on $M$.    The weighted scale invariant distance to the space of cones is the function
\begin{equation}
\cC(t)=\cC_x (t)=\int _M\Theta_{d_M(x,y)}(x)\,H(x,y,t)\,dy\, .
\end{equation}
Likewise we define the weighted $L^{\alpha}$ scale invariant distance by
\begin{equation}
\cC_{\alpha} (t)=\cC_{\alpha, x} (t)=\left(\int_M \Theta^{\alpha}_{d_M(x,y)}(x)\,H (x,y,t)\,dy\right)^{\frac{1}{\alpha}}\, .
\end{equation}
\end{Def}

Note that if $M$ has Euclidean volume growth, then by \cite{LY}, \cite{LTW}, there exists a constant $C(n,\V_M)>0$ such that
\begin{align}
\Theta_{\sqrt{t}}\leq C\,\cC(t)\, .
\end{align}
In fact, by the Cauchy-Schwarz inequality for $\alpha\geq 1$
\begin{equation}
\cC(t)\leq \cC_{\alpha}(t)\, .
\end{equation}

\subsection{Bounding the distance to cones} 

From Theorem \ref{t:fund} we get (see \cite{CM4} for more details):

\begin{Thm}  \label{t:fund2}
Given $\epsilon>0$, there exist $C=C(\epsilon,n,\V_M)>0$  and $c=c(n,\V_m)>1$ such that if $M$ is an $n$-dimensional manifold with nonnegative Ricci curvature and $h=-\log H$, where $H=H_x$ is the heat kernel and $x\in M$ is fixed, then for $t$ sufficiently large
\begin{equation}
\cC^2 (t)\leq C\, t^2 \int_M \left| \Hess_h-\frac{1}{2t}g\right|^2\,H\, .
\end{equation}
\end{Thm}

\end{document}